\documentclass{article} 
\usepackage{comment}

\usepackage{amsmath,amsthm}     
\usepackage{graphicx}     
\usepackage{tikz-cd}
\usepackage{hyperref}
\usepackage{url}
\usepackage{amsfonts} 
\usepackage{array}
\usepackage{multicol}

\newcommand{\calD}{\ensuremath{\mathcal{D}}}
\newcommand{\calA}{\ensuremath{\mathcal{A}}}
\newcommand{\calS}{\ensuremath{\mathcal{S}}}

\newcommand{\calT}{\ensuremath{\mathcal{T}}}

\newcommand{\Z}{\ensuremath{\mathbb{Z}}}

\newcommand{\defword}[1]{{\em #1}}
\newcommand{\sfa}{\ensuremath{{\sf a}}}
\newcommand{\sfb}{\ensuremath{{\sf b} }}
\newcommand{\expand}{{Q} }

\newtheorem{theorem}{Theorem}
\theoremstyle{definition}
\newtheorem{definition}[theorem]{Definition}

\newtheorem{proposition}[theorem]{Proposition}

\newtheorem{lemma}[theorem]{Lemma}

\begin{document}

\title{(Don't) Mind the Gap\\
\Large{Complexity of Gapped Digit Substitutions}}
\author{Natalie Priebe Frank\\  
Vassar College\\    
Poughkeepsie, NY 12604\\                
\href{mailto:nafrank@vassar.edu}{nafrank@vassar.edu}\\
May Mei\\  
Denison University\\    
Granville, OH 43023\\                
\href{mailto:meim@denison.edu}{meim@denison.edu}\\
Kitty Yang\\  
University of North Carolina at Asheville\\    
Asheville, NC 28804\\                
\href{mailto:kyang2@unca.edu}{kyang2@unca.edu}\\
}
\maketitle

\section{Introduction}\label{S:intro}
Infinite sequences are of tremendous theoretical and practical importance, and in the Information Age sequences of ${\sf 0}$s and ${\sf 1}$s are of particular interest.
Here is a good way to generate a binary sequence with many applications. Begin with the \defword{substitution rule}
\begin{center}
replace each ${\sf 0}$ with ${\sf 01}$ and ${\sf 1}$ with ${\sf 10}$.
\end{center}
Starting with ${\sf 0}$, we repeatedly applying the rule and concatenating the results, yields
\begin{equation}
 \label{eqn.TM.def}   
\sf{0} \mapsto {\sf 01} \mapsto {\sf 0110} \mapsto {\sf 01101001} \mapsto {\sf 0110100110010110} \mapsto \ldots. \end{equation}
Since we replace each letter with a word of length 2, at level $k$ we get a word of length $2^k$. The part of the word that was determined at level $k$ does not change in subsequent levels, and the lengths of words grow without bound, so in the limit we get an infinite sequence. A sequence generated by a substitution in this manner is known as a \defword{substitution sequence}.

The particular sequence generated from (\ref{eqn.TM.def}) is known by many names because of its wide-ranging applications. It is known as the \defword{parity sequence} because the $k^\text{th}$ term is precisely the parity of the number of 1s that occur in the binary representation of $k$ and can be used as a simple error detecting code. This sequence is also known as the \defword{fair-share sequence} because of its applications in fair division problems. In mathematical circles it is most well known as the \defword{Thue–Morse sequence}, named after some of the mathematicians who discovered it and that is what we will call it. A substitution sequence is rather special --- if we apply the substitution rule to it, it does not change, and we say is a \defword{fixed point} of the substitution.

We decided to write down all the words of lengths 2, 3, and 4 in the Thue-Morse sequence, and the complete list is given below. Notice that only 10 of the $2^4=16$ possible length-4 words of ${\sf0}$s and ${\sf 1}$s appear.
\[
\begin{array}{ | w{c}{1cm} | w{c}{1cm}| w{c}{1cm} w{c}{1cm}| } 
    \sf 00 & \sf 001 & \sf 0010 & \sf 1001 \\
    \sf 01 & \sf 010 & \sf 0011 & \sf 1010 \\
    \sf 10 & \sf 011 & \sf 0100 & \sf 1011 \\
    \sf 11 & \sf 100 & \sf 0101 & \sf 1100 \\
           & \sf 101 & \sf 0110 & \sf 1101 \\
           & \sf 110\\
\end{array}\]

Given a sequence $\omega$, its \defword{complexity function} $p_\omega(n)$ counts the number of distinct words of length $n$ appearing in $\omega$. Based on our computations, with $\omega=TM$ as the infinite sequence generated from (\ref{eqn.TM.def}), $p_{\text TM}(4)=10$. In general, finding $p_{TM}(n)$ seems highly doable for small values of $n$, and we can always program a computer to at least give a lower bound for the complexity function by counting the number of words in a very long finite string. But how can we figure out all the values of the complexity function? How does the sequence being generated by a substitution factor into the situation? We will explore these questions in what follows. Then, we will exhibit a relatively new generalization of substitutions, ones that have \defword{gaps}, and adapt our tools to study the complexity of those as well.

\section{Complexity via right special words} \label{sec:complexity}

We have computed a few values of the complexity function for the Thue-Morse sequence, but we would like a formula for the whole thing. We cannot actually list all possible words that can ever appear anywhere in the sequence. The key idea that allows for the computation of complexity functions is \defword{special words}. While the Thue-Morse sequence is a binary sequence, there is no reason sequences cannot contain more symbols than just ${\sf0}$s and ${\sf1}$s, so we state the next definition with a general \defword{alphabet $\calA$} of symbols in mind.

\begin{definition}
    We say that a word $w$ in a sequence $\omega$ is \defword{right special} if there are distinct symbols $a \ne b$ for which both $wa$ and $wb$ appear in the sequence. Let $s_\omega(n)$ denote the number of right special words of length $n$ that appear in $\omega$. When the context is clear, we drop the subscript $\omega$.
\end{definition}
For instance, since ${\sf 01}$ and ${\sf 00}$ are both words in the Thue-Morse sequence, the word ${\sf 0}$ is right special. If the sequence is constructed from exactly two symbols, the complexity function is related to the number of right special words via the following:

\begin{proposition}[\cite{Cas97}, Proposition 3.1]\label{prop:special}
    If $\omega$ is a binary sequence, then $p(n+1) = p(n)+s(n)$.
\end{proposition}

Let us offer the following heuristic as to why Proposition \ref{prop:special} holds: if a word is not right special, then it can be continued in only one way. But if a word \emph{is} right special, then it can be continued in exactly two ways. 

An interesting property of right special words is that if a right special word $w$ can be written as the concatenation of two words $w=uv$, then $v$ must also be a right special word. Thus, we can build successively longer right special words by seeing how we can extend right special words on the left. We can visualize this process using a \defword{right special tree}, an infinite tree whose vertices are labeled by the symbols in $\omega$. Here is how to interpret the special tree in Figure~\ref{fig:TMtree}: Starting from any location and reading from left to right until the rightmost column will yield a right special word. For convenience, we have numbered each column with the length of the word associated with a given node.

Now, let us construct the right special tree for Thue-Morse, making sure it contains all possible right special words. We have already noted that ${\sf 0}$ is right special because it can be extended to the right by both ${\sf 0}$ and ${\sf 1}$, so it is in the first column. The 2-letter word ${\sf 10}$ is also right special, so there is a ${\sf 1}$ in the second column connected to the ${\sf 0}$ in the first column. 
And similarly, extending ${\sf 1}$ on the left yields the right special word ${\sf 01}$. Continuing, we find that ${\sf 010}$ and ${\sf 110}$ are right special, allowing us to fill in the third column. It might seem a little odd that the top node of column 3 does not have a left extension. This is because while ${\sf 010}$ is right special, if we extend to the left by either letter the result is not right special.

\begin{figure}[h]
\small{
\[\begin{tikzcd}[cramped]
	\dots & 10 & 9 & 8 & 7 & 6 & 5 & 4 & 3 & 2 & 1 \\
	\dots & {\sf{1}} & {\sf{0}} &&& {\sf{1}} & {\sf{0}} && {\sf{0}} \\
	\dots & {\sf{0}} & {\sf{1}} & {\sf{1}} & {\sf{0}} & {\sf{0}} & {\sf{1}} & {\sf{0}} & {\sf{1}} & {\sf{1}} & {\sf{0}} \\
	\dots & {\sf{1}} & {\sf{0}} & {\sf{0}} & {\sf{1}} & {\sf{1}} & {\sf{0}} & {\sf{1}} & {\sf{0}} & {\sf{0}} & {\sf{1}} \\
	\dots & {\sf{0}} & {\sf{1}} &&& {\sf{0}} & {\sf{1}} && {\sf{1}}
	\arrow[no head, from=3-11, to=3-10]
	\arrow[no head, from=4-11, to=4-10]
	\arrow[no head, from=3-10, to=3-9]
	\arrow[no head, from=4-9, to=4-10]
	\arrow[no head, from=3-10, to=2-9]
	\arrow[no head, from=4-10, to=5-9]
	\arrow[no head, from=3-9, to=3-8]
	\arrow[no head, from=4-9, to=4-8]
	\arrow[no head, from=3-8, to=2-7]
	\arrow[no head, from=3-8, to=3-7]
	\arrow[no head, from=4-8, to=4-7]
	\arrow[no head, from=4-8, to=5-7]
	\arrow[no head, from=2-7, to=2-6]
	\arrow[no head, from=3-7, to=3-6]
	\arrow[no head, from=4-7, to=4-6]
	\arrow[no head, from=5-7, to=5-6]
	\arrow[no head, from=3-6, to=3-5]
	\arrow[no head, from=4-6, to=4-5]
	\arrow[no head, from=3-5, to=3-4]
	\arrow[no head, from=4-5, to=4-4]
	\arrow[no head, from=3-4, to=2-3]
	\arrow[no head, from=3-4, to=3-3]
	\arrow[no head, from=4-4, to=4-3]
	\arrow[no head, from=4-4, to=5-3]
	\arrow[no head, from=2-3, to=2-2]
	\arrow[no head, from=3-3, to=3-2]
	\arrow[no head, from=4-3, to=4-2]
	\arrow[no head, from=5-3, to=5-2]
\end{tikzcd}\]}
\caption{Right special tree for the Thue-Morse substitution from \cite{Cas97}.}\label{fig:TMtree}
\end{figure}
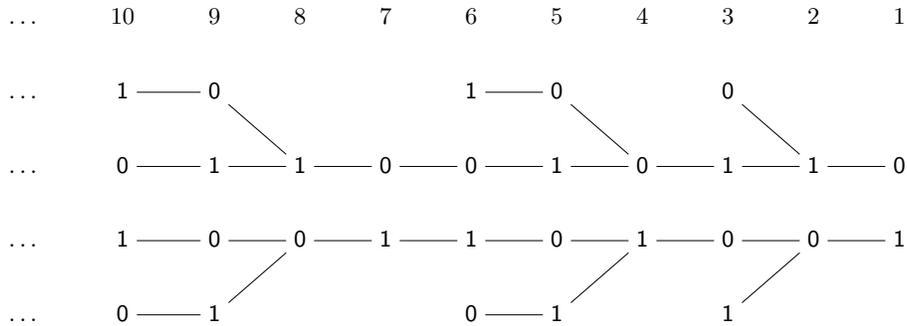

There is enough of the tree to suspect that the branching pattern repeats, maybe with more space in-between. There are several reasons for this pattern. Consider the right special word ${\sf 0110}$. This word is right special since ${\sf 01100}$ and ${\sf 01101}$ both appear in the Thue-Morse sequence. Applying the substitution to ${\sf 0110}$ results in ${\sf 01101001}$, which is now a length-8 right special word. You may want to verify for yourself that this holds for some other right special words; in fact, this is true for all right special words in the Thue-Morse sequence:

\begin{proposition}
    If $w$ is a right special word in the Thue-Morse sequence, then the image of $w$ under the substitution is also right special.
\end{proposition}

Since applying the substitution generates words twice as long, the repetition in the tree structure might take longer to appear. While there is not space to prove it here, it is true that once we compute enough of the tree by hand, the rest can be propagated by applying the substitution. If you have the patience to try it, you will be rewarded by discovering the pattern for yourself. The number of nodes in the $n$th column of the tree gives us $s(n)$, and it turns out that
$$s_{\text{TM}}(n)=\begin{cases}
  1,  & n=0 \\
  2, & 0<n\leq2\\
  4, & 2 \cdot 2^m<n\leq 3\cdot 2^m\\
  2, & 3 \cdot 2^m<n\leq 4 \cdot 2^m.
\end{cases}$$

Now that we know the right special function for Thue-Morse, it in combination with Proposition \ref{prop:special} allow us to derive a formula for the complexity function $p_{TM}$.
We know that the Thue-Morse sequence involves two letters, so $p_{TM}(1)=2$. By Proposition~\ref{prop:special}, $p_{TM}(2)=p_{TM}(1)+s_{TM}(1)=2+2=4$. Continuing with the same recursive relationship: $p_{TM}(3)=p_{TM}(2)+s_{TM}(2)=4+2=6$. Repeating one more time, $p_{TM}(4)=p_{TM}(3)+s_{TM}(3)=6+4=10$. Indeed, these initial calculations agree with the list we laid out in the introduction.
\[(p_{TM}(2), p_{TM}(3), p_{TM}(4), \dots , p_{TM}(10))=(4,6,10,12,16,20,22,24,28).\]
From $s_{TM}(n)$ and Proposition \ref{prop:special}, we obtain recursively that
$$p_{\text{TM}}(n) = \begin{cases}
  1,  & n=0 \\
  2, & n=1 \\
  4, & n=2 \\
  4n-2\cdot2^m-4, & 2\cdot 2^m < n \leq 3 \cdot 2^m, m \geq 1\\
  2n+4\cdot2^m-2, & 3\cdot 2^m < n \leq 4 \cdot 2^m, m \geq 1.
\end{cases}$$

\section{Substitutions with gaps}

Now, it is time to explore the generalization we have in mind. The Thue-Morse substitution is an example of a \defword{constant-length} substitution rule: one where letters in the alphabet are replaced by words of the same length. A straightforward generalization is to allow the lengths of the replacement words to vary. The most famous example of that is the Fibonacci substitution: 
\begin{center}
Replace ${\sf a}$ with ${\sf ab}$ and replace ${\sf b}$ with ${\sf a}$.
\end{center}
The lengths of iterated words are Fibonacci numbers:
\[\sf{a} \mapsto {\sf ab} \mapsto {\sf aba} \mapsto {\sf abaab} \mapsto {\sf abaababa} \mapsto \ldots. \]

A great deal of current research centers on non-constant-length substitutions, which typically have non-integer expansion factors governing the growth rate of words. We want to show a different generalization where the expansion factor is still an integer. In our generalization, letters can be replaced by non-contiguous words, or words with \defword{gaps}.

In order to define a gapped substitution rule, we will need to keep track of the position of symbols, as well as the symbol itself. To help distinguish the position from the symbol, we use the letters {\sfa}  and {\sfb} instead of ${\sf 0}$ and ${\sf 1}$. The pair $(n,\sf c)$, with $n \in \mathbb{Z}$ and $\sf c \in \{\sfa, \sfb\}$, indicates that the symbol $\sf c$ occurs at position $n$. Consider a substitution rule is as follows:
\begin{center}
Replace $(0, \sfa)$ with $\{(-1,\sfb),(0,\sfa),(4,\sfb)\}$ and 
$(0,\sfb)$ with $\{(-1,\sfa),(0,\sfb),(4,\sfa)\}$.
\end{center}

\begin{figure}[ht]
    \centering
    \includegraphics[width=1in]{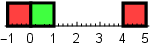}
    \caption{The image of $(0, \sfa)$, with letters depicted in \textcolor{red}{red} and  \textcolor{green}{green}.}
\end{figure}
\noindent We also need a rule that tells us what to do to symbols not located at position $0$. For a symbol that occurs at any nonzero position, we shift the result by an inflation factor of 3. That is to say, replace $(\ell, \sfa)$ with $\{(3\ell -1,\sfb),(3\ell,\sfa),(3\ell + 4,\sfb)\}$ and $(\ell,\sfb)$ with $\{(3\ell -1,\sfa),(3\ell,\sfb),(3\ell + 4,\sfa)\}$. Since we can no longer concatenate words, we apply the substitution to each symbol and take the union of all the images. Put differently, for $\ell \in \Z$ we denote $\calS$ as follows:

\[
\calS(\ell, \sf{c}) = \begin{cases}
  \{(3\ell-1,\sfb),(3\ell,\sfa),(3\ell+4,\sfb)\} & \text{ if } \,\sf{c}=\sfa\\  
  \{(3\ell-1,\sfa),(3\ell,\sfb),(3\ell+4,\sfa)\} & 
\text{ if } \,\sf{c}=\sfb
\end{cases}
\]
For brevity, we sometimes replace $\calS(0,\sfa)$ with $\calS(\sfa)$ and $\calS(0,\sfb)$ with $\calS(\sfb)$. Iterating $\calS$, we obtain 
\begin{align*}
\calS^{2}(\sfa)&=\{(-4,\sfa),(-3,\sfb),(1,\sfa)\} \cup 
\{(-1,\sfb),(0,\sfa),(4,\sfb)\} \cup 
\{(11,\sfa),(12,\sfb),(16,\sfa)\}\\
&=\{(-4,\sfa),(-3,\sfb),(-1,\sfb),(0,\sfa),(1,\sfa),(4,\sfb),(11,\sfa),(12,\sfb),(16,\sfa)\}.
\end{align*}
\begin{figure}[ht]
    \centering
    \includegraphics[width=\textwidth]{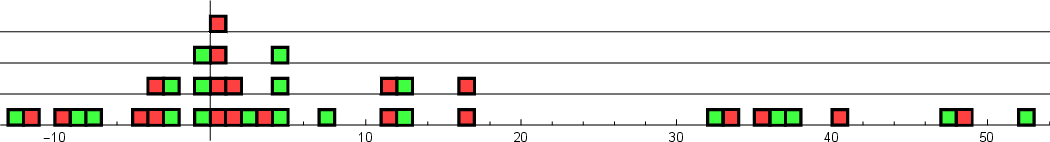}
    \caption{Here are the first few levels of the substitution. From top to bottom: $(\sfa)$, $\calS(\sfa)$, $\calS^2(\sfa)$, and $\calS^3(\sfa)$}
    \label{fig:gappedsupertiles}
\end{figure}

Is this substitution with gaps well-defined, or might we accidentally specify a symbol in two different ways at the same level? Since $-1 \equiv 2\mod 3$, $0 \equiv 0\mod 3$, and $4 \equiv 1\mod 3$, we can write any integer uniquely as $3\ell + d$, where $\ell \in \mathbb{Z}$ and $d = -1, 0$, or $4$. This ensures any symbol that appears in $\calS^k(0,\sf{c})$ can only be an element of the image from exactly one symbol from level $\calS^{k-1}(0,\sf{c})$. 

Another issue that could arise with a gapped substitution is whether we can even generate arbitrarily long strings of symbols. This was key in obtaining a fixed point of the Thue-Morse substitution. Figure \ref{fig:gappedsupertiles} shows that while each level introduces more gaps, it also fills in locations near 0. In $\calS(\sfa)$, there is a word of length two starting at location $-1$; it turns into words of length three and six in $\calS^2( \sfa)$ and $\calS^3(\sfa)$, respectively. We are going to prove that this so-called ``central patch" grows without bound. As long as subsequent iterations don't change previously defined letters (in the same way that the initial segment of the Thue-Morse sequence does not change as the sequence lengthens), we only need to worry about the \defword{support} of each level, or the locations where letters have been defined.

To help with notation, we will use the notion of a \defword{digit system} $(\expand, \calD)$. First, choose an integer $\expand$ greater than $1$ to serve as the \defword{expansion constant} of the system. We need a set $\calD$ of integers so that every $k \in \Z$ is uniquely expressed as $\expand j + d$ where $j \in \Z$ and $d \in \calD$. Thus we require the \defword{digit set} $\calD$ to have this property, i.e., to be a full set of coset representatives for $\Z\pmod \expand$. In the gapped digit example we've been discussing so far, $\expand =3$ and $\calD=\{-1,0,4\}$, whereas the Thue-Morse substitution has $\expand=2$ and $\calD=\{0,1\}$. When the digit set is contiguous, we'll refer to the substitution as \defword{constant-length}. Otherwise, we'll say the substitution is \defword{gapped}. 

Let $\calD_k$ denote the set of locations of symbols in $\calS^k(\sfa)$. Let $\calS$ be a substitution with $\expand =3$ and $\calD=\{-1,0,4\}$. We take the gapless word starting from the $-1$ location in $\calS^{k}(\sfa)$ to be the \defword{central patch} of $\calD_k$. The next lemma allows us to define the substitution sequence associated with $\calS$ to be the limit of the words on these central patches. 

\begin{lemma}[Central Patch Lemma]\label{L:CentralPatchLemma}
Let $\expand =3$ and $\calD=\{-1,0,4\}$. For $k \ge 1$, the central patch of $\calD_k$ is the set $\{-1, 0, 1, \ldots, 3^{k-1}\cdot 2 - \frac{3^{k} - 3}{2} - 2\}$.
\end{lemma}

\begin{proof}
Algebra note: In this proof, multiplying a set by a number means multiplying each element by that number, and adding two sets yields the set of all possible sums of pairs of elements, one from each set. We know that $\{-1, 0\} \subset \calD = \calD_1$ and  $-1 \in 3\calD + \calD = \calD_2$, and that $-2 \in \{-7, -6, -2\}$, a nontrivial translate of $\calD$, as well as in $\calS(\{-7,-6,-2\})$. Thus, $-1$ is in every central patch while $-2$ is in none of them. Applying $\calS$ to the symbol located at $-1$ results in a symbol at $ 3(-1)+4 = 1$, which means that $\calD_2$ contains the set $\{-1, 0, 1\}$. 

More generally, if $\calD_k$ contains $j+1$ consecutive digits, where $j \geq 1$, then $\calD_{k+1}$ contains $3j$ consecutive digits (one can verify this by applying the digit substitution to a string of integers to see how the different images fit together). Thus, the central patch of  $\calD_k$ contains consecutive digits starting at $-1$. Let $d_k$ denote the length of the central patch of $\calD_k$.

We can rewrite the statements above as the recursive formula:
\begin{align*}
    d_1 &= 2\\
    d_{k+1} &= 3(d_k -1).
\end{align*}

The geometric sum formula shows that $$d_k=3^{k-1}\cdot 2 - \left(\sum_{n=1}^{k-1} 3^n\right)=3^{k-1}\cdot 2 - \frac{3^{k}-3}{2}.$$
\end{proof}

Observe that since $\calS(\sfa)$ always has $\sfa$ in the 0 position, the central patch of level $k$ always agrees with the central patch of level $k+1$. This is analogous to the fact that the level $k$ Thue-Morse word is always a prefix of the level $k+1$ Thue-Morse word. 

We say that a word of the form $\calS^k(\sfa)$ is a \defword{level-$k$ supertile}. Let's think about the Thue-Morse sequence. We can think of a $\sf 0110$ as the concatenation of two level-1 supertiles: $\sf 01$ and $\sf10$. Each level-$k$ supertile is adjacent to two other supertiles: one on the right and one on the left. Figure~\ref{fig:supertileadjacency2} illustrates this for level-1 supertiles for $\calS$. Applying $\calS$ to the 5-letter word $w=(-2,a),(-1,b),(0,c),(1,d),(2,e)$, notice that $\calS(0,c)$, shown in bold, is adjacent to four level-1 supertiles.  Since every letter in $\calS(0,c)$ is contained within a connected 5-letter word of $\calS(w)$, the level-1 supertile neighbors of $\calS^2(0,c)$ are $\calS^2(-2,a), \calS^2(-1,b), \calS^2(1,d),$ and $\calS^2(2,e)$. 

\begin{figure}[ht]
    \centering    \includegraphics[width=.5\textwidth]{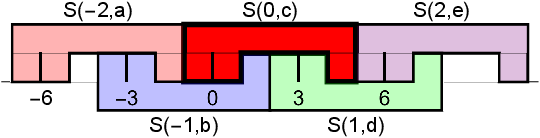}
    \caption{The adjacency structure in the substitution of a length 5 word.}
    \label{fig:supertileadjacency2}
\end{figure}

We wondered about how the complexity of a gapped digit substitution compares to the complexity of a constant-length substitution with the same expansion factor. So, we decided to investigate digit substitutions with letters $\sfa, \sfb$ and $\expand = 3$ with digit sets $\calD = \{-1, 0, 4\}$ and $\calD' = \{-1, 0, 1\}$. For each sextuple given by the first two columns, we define 
\begin{align*}
    \calS(\sfa)=\{(-1,a_1),(0,a_2),(4,a_3)\},&\ \calS(\sfb)=\{(-1,b_1),(0,b_2),(4,b_3)\},\\
    \calS'(\sfa)=\{(-1,a_1),(0,a_2),(1,a_3)\},&\ \calS'(\sfb)=\{(-1,b_1),(0,b_2),(1,b_3)\}.
\end{align*}
Figure~\ref{fig:complexitytable} shows $p(n)$ and $p'(n)$ for $n = 2, 3, \dots, 10$ for a selection of ``interesting'' substitutions:
those with $\calS(\sfa)_0 =\sfa$ and such that $\calS(\sfa)$ contains at least one $\sfb$ and $\calS(\sfb)$ contains at least one $\sfa.$ (We used the computer to figure them out). Of course, we'd like to be able to determine the full complexity functions, possibly using a right special tree again. However, the gaps in the substitution confuse the notion of adjacency and make propagating the tree ad infinitum more difficult. It turns out that with a correct understanding of ``right special,'' we can adapt the earlier techniques to propagate the tree and compute the complexity.

Next, we'll obtain a way to, in essence, remove the gaps via a recoding process. Then we'll be able to use a right special tree to compute the complexity function following the procedure that worked for Thue-Morse.

\begin{figure}[tp]\small{
\[\begin{array}{|c|c||c|c|}
\hline
a_1,a_2,a_3 & b_1, b_2, b_3 & (p(2), p(3),\dots , p(10))&(p'(2), p'(3),\dots , p'(10))\\
\hline
\hline
\sf a, a, b & \sf a, a, b & (3, 3, 3, 3, 3, 3, 3, 3, 3)        & (3, 3, 3, 3, 3, 3, 3, 3, 3)\\
\sf a, a, b & \sf a, b, a & (4, 7, 13, 18, 25, 32, 39, 44, 49) & (3, 5, 7, 9, 11, 13, 15, 17, 19)\\
\sf a, a, b & \sf a, b, b & (4, 6, 8, 10, 12, 14, 16, 18, 20)  & (4, 6, 8, 10, 12, 14, 16, 18, 20)\\
\sf a, a, b & \sf b, a, a & (4, 7, 13, 22, 29, 37, 47, 57, 66) & (4, 6, 9, 12, 15, 18, 21, 24, 27)\\
\sf a, a, b & \sf b, a, b & (4, 6, 8, 10, 12, 14, 16, 18, 20)  & (4, 6, 8, 10, 12, 14, 16, 18, 20)\\
\sf a, a, b & \sf b, b, a & (4, 8, 14, 20, 28, 36, 44, 52, 60) & (4, 8, 12, 16, 20, 24, 28, 32, 36)\\
\sf b, a, a & \sf a, a, b & (4, 7, 13, 22, 29, 37, 47, 57, 66) & (4, 6, 9, 12, 15, 18, 21, 24, 27)\\
\sf b, a, a & \sf a, b, a & (3, 5, 7, 9, 11, 13, 15, 17, 19)   & (3, 5, 7, 9, 11, 13, 15, 17, 19)\\
\sf b, a, a & \sf a, b, b & (2, 2, 2, 2, 2, 2, 2, 2, 2)        & (4, 8, 12, 16, 20, 24, 28, 32, 36)\\
\sf b, a, a & \sf b, a, a & (3, 3, 3, 3, 3, 3, 3, 3, 3)        & (3, 3, 3, 3, 3, 3, 3, 3, 3)\\
\sf b, a, a & \sf b, a, b & (4, 6, 8, 10, 12, 14, 16, 18, 20)  & (4, 6, 8, 10, 12, 14, 16, 18, 20)\\
\sf b, a, a & \sf b, b, a & (4, 6, 8, 10, 12, 14, 16, 18, 20)  & (4, 6, 8, 10, 12, 14, 16, 18, 20)\\
\sf b, a, b & \sf a, a, b & (4, 6, 8, 10, 12, 14, 16, 18, 20)  & (4, 6, 8, 10, 12, 14, 16, 18, 20)\\
\sf b, a, b & \sf a, b, a & (4, 8, 14, 20, 28, 36, 44, 52, 60) & (2, 2, 2, 2, 2, 2, 2, 2, 2)\\
\sf b, a, b & \sf a, b, b & (3, 5, 7, 9, 11, 13, 15, 17, 19)   & (3, 5, 7, 9, 11, 13, 15, 17, 19)\\
\sf b, a, b & \sf b, a, a & (4, 6, 8, 10, 12, 14, 16, 18, 20)  & (4, 6, 8, 10, 12, 14, 16, 18, 20)\\
\sf b, a, b & \sf b, a, b & (3, 3, 3, 3, 3, 3, 3, 3, 3)        & (3, 3, 3, 3, 3, 3, 3, 3, 3)\\
\sf b, a, b & \sf b, b, a & (4, 7, 13, 18, 25, 32, 39, 44, 49) & (3, 5, 7, 9, 11, 13, 15, 17, 19)\\
\hline
\end{array}\]}
\caption{Complexities for digit substitutions with $\expand = 3$ and digit sets $\calD = \{-1, 0, 4\}$ (column 3) and $\calD' = \{-1, 0, 1\}$ (column 4).}
\label{fig:complexitytable}
\end{figure}

\section{Recoding Digit Substitutions}
\label{S:conjugacy}
Now that we've highlighted the differences between constant-length and gapped digit substitutions, let's talk about how they are the same. Namely, we show how to \defword{recode} digit substitution sequences into constant-length substitution sequences on a larger \defword{alphabet}, which is the collection of symbols in the sequence. The recoding process is a special example of a \defword{sliding block code}. A sliding block code defines a new sequence on a different set of symbols. It takes a small window of one sequence and defines symbols of a second sequence one new symbol at a time by ``sliding" the window to the right.

Suppose $\omega$ is a sequence on a finite alphabet $\calA$ and choose some $N>1$ to be the length of the window. We make a new alphabet $\calA^N$ which is the set of all $2^{|A|}$ vectors of length $N$ whose elements are from $\calA$. Note that we will not use all the elements in $\calA^N$; only length-$N$ words that appear in $\omega$ will be used. We then recode $\omega$ to a sequence $\omega^{[N]}$ on the new alphabet $\calA^N$.
In general,
 \[ (\omega^{[N]})_i =
 \begin{bmatrix}
     \omega_{N+i-2}\\
      \vdots\\
            \omega_{i-1}
 \end{bmatrix}
 \quad \text{ and } \quad
\omega^{[N]}= 
\begin{bmatrix}
 \omega_{N-1}\\
 \vdots\\
 \omega_{0}
\end{bmatrix}
\begin{bmatrix}
 \omega_{N}\\
 \vdots\\
 \omega_{1}
\end{bmatrix}
\begin{bmatrix}
 \omega_{N+1}\\
 \vdots\\
 \omega_{2}
\end{bmatrix}
\cdots
\]
In a little while we will explicitly compute the higher-block recoding on the gapped digit substitution example.

The sequence one obtains from a higher-block recoding has same complexity function as the original sequence, but shifted a bit according to the size of the window:

\begin{proposition} \label{prop:complexity}
    Let $\omega$ be a sequence and $\omega^{[N]}$ be its $N$th higher block recoding. If $p_\omega(n)$ is the complexity function for $\omega$, then the complexity function for $\omega^{[N]}$ is \[p_{\omega^{[N]}}(n)=p_\omega(n+N-1).\]
\end{proposition}

\begin{proof}
Each word of length $n$ in $\omega^{[N]}$ is fully determined by a word of length $n+N-1$ in $\omega$.
\end{proof}

Now, we assume that $\calS$ is a fixed digit substitution on $\calA=\{\sfa, \sfb\}$ on the digit system $\expand = 3$ and $\calD=\{-1, 0, 4\}$, 
and let $\omega$ be an associated substitution sequence. 
We will show that $\calS$ has, as its 3rd higher block representation, a constant-length substitution sequence given by $\calT$ on the symbols $\calA^{[3]}$. Because the $\calS$ has an expansion factor of 3, we'll see that so does $\calT$. 

It is convenient to introduce and clarify some notation. Let $\omega_{[m,m+n]}$ denote the subword $\omega_m \omega_{m+1}\dots \omega_{m+n}$. We denote the symbols for the higher-block recoding by $\calA^{[3]}$, which arise from length-3 words $abc \in \calA^3$ which appear in $\omega$. We distinguish between the words in the original sequence and symbols in the recoding by referring to the latter as vectors: $(abc)^{[3]} \in \calA^{[3]}$. To find $\calT((abc)^{[3]})$ we apply $\calS$ to $\left(\calS^k(d)\right)_{[n-1,n+1]} \subset \calS^{k+1}(d)$, where $S^k(d)$ is a supertile that contains $abc$, centered at coordinate $n$.

As shown in Figure~\ref{fig:3adjacent2}, the patch $[3n-1,3n+4] \subseteq \mathbb{Z}$ is  determined by $\calS(abc)$. We only need the 5-letter word on the domain $[3n -1, 3n + 3]$ to recode to the 3-letter word in $\calA^{[3]}$ that defines the substitution of $(abc)^{[3]}$. Notice that the substitution of $b$ is centered at $3n$ and the substitutions of $a$ and $c$ are centered at $3(n-1)$ and $3(n+1)$, thus the substitution on $(abc)^{[3]}$ should have length 3 to make room for the substitutions of the letters $(*ab)^{[3]}$ and $(bc*)^{[3]}$. If we apply this process to a different supertile that contains $abc$, we get the same 5-letter word. Thus, this substitution rule is well-defined. This also illustrates how to read off the substitution for $(abc)^{[3]}$. We define
\begin{equation*}\label{eq:calT}
    \calT ((abc)^{[3]}) := (\calS^{k+1}(d)_{[3n-1,3n+3]})^{[3]},
\end{equation*}
where $abc \in \mathcal{A}^3$ is a word that appears in a supertile $S^k(d)$.

\begin{figure}[ht]
    \centering
    \includegraphics[width=.5\textwidth]{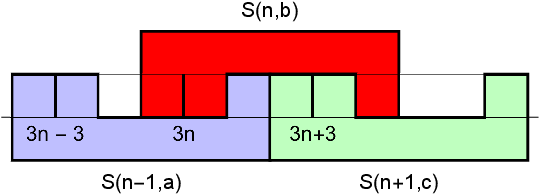}
    \caption{The image under $\calS$ of a connected 3-letter word contains a 6-letter word.}
    \label{fig:3adjacent2}
\end{figure}

We are ready to explicitly compute the recoding of a gapped digit substitution and use it to compute the complexity.
Let us return to the the gapped digit substitution we worked with previously:
\begin{align*}
\calS(\sfa):=\calS(0,\sfa)&=\{(-1,\sfb),(0,\sfa),(4,\sfb)\}\\
\calS(\sfb):=\calS(0,\sfb)&=\{(-1,\sfa),(0,\sfb),(4,\sfa)\}.
\end{align*}
and explicitly show how to recode it. We chose this example because of the large difference in complexity between the last two columns of Figure~\ref{fig:complexitytable} (corresponding to ${\sf a \to bab, b \to aba},$ which produces a periodic sequence of repeating $\sf ab$s).

For standardized and more efficient notation let's use
\begin{align*}
    {\sf 1 \mapsto aaa} =\begin{bmatrix}
            \sfa\\
            \sfa\\
            \sfa
    \end{bmatrix} , 
     {\sf 2 \mapsto baa} =\begin{bmatrix}
            \sfb\\
            \sfa\\
            \sfa
    \end{bmatrix} , 
     {\sf 3 \mapsto aba} =\begin{bmatrix}
            \sfa\\
            \sfb\\
            \sfa
    \end{bmatrix} , 
     {\sf 4 \mapsto bba} =\begin{bmatrix}
            \sfb\\
            \sfb\\
            \sfa
    \end{bmatrix} , \\
        {\sf 5 \mapsto aab} =\begin{bmatrix}
            \sfa\\
            \sfa\\
            \sfb
    \end{bmatrix} , 
     {\sf 6 \mapsto bab} =\begin{bmatrix}
            \sfb\\
            \sfa\\
            \sfb
    \end{bmatrix} , 
     {\sf 7 \mapsto abb} =\begin{bmatrix}
            \sfa\\
            \sfb\\
            \sfb
    \end{bmatrix} , 
     {\sf 8 \mapsto bbb} =\begin{bmatrix}
            \sfb\\
            \sfb\\
            \sfb
    \end{bmatrix}.
\end{align*}

The central patch of $\calS^k(\sf a)$ is fixed under substitution. The central patch is $\sf{baabababbbaaab\ldots}$, so the recoding of the fixed point is
\begin{equation*}
\sf523636487512 \cdots
\end{equation*}
The central patch of the original substitution was fixed, so this is a fixed point of the recoded substitution. That means we can simply read off the substitution $\calT$ by looking at length-3 words, starting with $5$: 
\begin{equation*}
\calT({\sf 5}) = {\sf 523}, \calT({\sf 2}) = {\sf 636}, \calT({\sf 3}) = {\sf 487}, \calT({\sf 6}) = {\sf 512}, \cdots
\end{equation*}

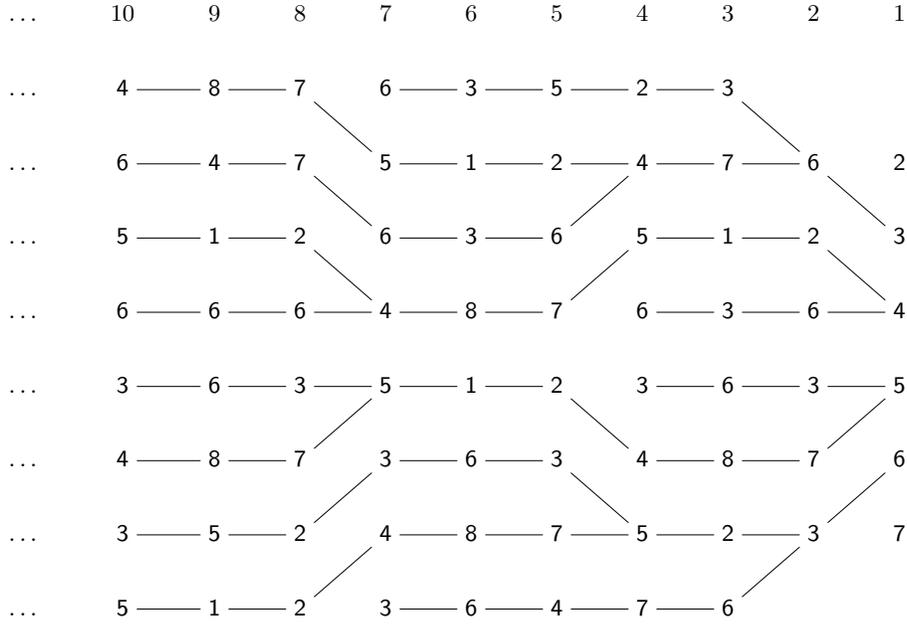
\begin{figure}[ht]
\small{
\[\begin{tikzcd}[cramped]
	\dots & 10 & 9 & 8 & 7 & 6 & 5 & 4 & 3 & 2 & 1 \\
	\dots & {\sf 4} & {\sf 8} & {\sf 7} & {\sf 6} & {\sf 3} & {\sf 5} & {\sf 2} & {\sf 3} \\
	\dots & {\sf 6} & {\sf 4} & {\sf 7} & {\sf 5} & {\sf 1} & {\sf 2} & {\sf 4} & {\sf 7} & {\sf 6} & {\sf 2} \\
	\dots & {\sf 5} & {\sf 1} & {\sf 2} & {\sf 6} & {\sf 3} & {\sf 6} & {\sf 5} & {\sf 1} & {\sf 2} & {\sf 3} \\
	\dots & {\sf 6} & {\sf 6} & {\sf 6} & {\sf 4} & {\sf 8} & {\sf 7} & {\sf 6} & {\sf 3} & {\sf 6} & {\sf 4} \\
	\dots & {\sf 3} & {\sf 6} & {\sf 3} & {\sf 5} & {\sf 1} & {\sf 2} & {\sf 3} & {\sf 6} & {\sf 3} & {\sf 5} \\
	\dots & {\sf 4} & {\sf 8} & {\sf 7} & {\sf 3} & {\sf 6} & {\sf 3} & {\sf 4} & {\sf 8} & {\sf 7} & {\sf 6} \\
	\dots & {\sf 3} & {\sf 5} & {\sf 2} & {\sf 4} & {\sf 8} & {\sf 7} & \sf5 & {\sf 2} & {\sf 3} & {\sf 7} \\
	\dots & {\sf 5} & {\sf 1} & {\sf 2} & {\sf 3} & {\sf 6} & {\sf 4} & {\sf 7} & {\sf 6}
	\arrow[no head, from=7-11, to=8-10]
	\arrow[no head, from=6-11, to=6-10]
	\arrow[no head, from=6-11, to=7-10]
	\arrow[no head, from=4-10, to=5-11]
	\arrow[no head, from=5-10, to=5-11]
	\arrow[no head, from=3-10, to=4-11]
	\arrow[no head, from=2-9, to=3-10]
	\arrow[no head, from=3-9, to=3-10]
	\arrow[no head, from=4-9, to=4-10]
	\arrow[no head, from=5-9, to=5-10]
	\arrow[no head, from=6-9, to=6-10]
	\arrow[no head, from=7-9, to=7-10]
	\arrow[no head, from=8-9, to=8-10]
	\arrow[no head, from=9-9, to=8-10]
	\arrow[no head, from=2-8, to=2-9]
	\arrow[no head, from=3-8, to=3-9]
	\arrow[no head, from=4-8, to=4-9]
	\arrow[no head, from=5-8, to=5-9]
	\arrow[no head, from=6-8, to=6-9]
	\arrow[no head, from=9-8, to=9-9]
	\arrow[no head, from=8-8, to=8-9]
	\arrow[no head, from=7-8, to=7-9]
	\arrow[no head, from=2-7, to=2-8]
	\arrow[no head, from=3-7, to=3-8]
	\arrow[no head, from=4-7, to=3-8]
	\arrow[no head, from=5-7, to=4-8]
	\arrow[no head, from=6-7, to=7-8]
	\arrow[no head, from=7-7, to=8-8]
	\arrow[no head, from=8-7, to=8-8]
	\arrow[no head, from=9-7, to=9-8]
	\arrow[no head, from=2-6, to=2-7]
	\arrow[no head, from=3-6, to=3-7]
	\arrow[no head, from=4-6, to=4-7]
	\arrow[no head, from=5-6, to=5-7]
	\arrow[no head, from=6-6, to=6-7]
	\arrow[no head, from=7-6, to=7-7]
	\arrow[no head, from=8-6, to=8-7]
	\arrow[no head, from=9-6, to=9-7]
	\arrow[no head, from=2-5, to=2-6]
	\arrow[no head, from=3-5, to=3-6]
	\arrow[no head, from=4-5, to=4-6]
	\arrow[no head, from=5-5, to=5-6]
	\arrow[no head, from=6-5, to=6-6]
	\arrow[no head, from=7-5, to=7-6]
	\arrow[no head, from=8-5, to=8-6]
	\arrow[no head, from=9-5, to=9-6]
	\arrow[no head, from=5-4, to=5-5]
	\arrow[no head, from=6-4, to=6-5]
	\arrow[no head, from=4-4, to=5-5]
	\arrow[no head, from=8-4, to=7-5]
	\arrow[no head, from=9-4, to=8-5]
	\arrow[no head, from=6-5, to=7-4]
	\arrow[no head, from=3-5, to=2-4]
	\arrow[no head, from=3-4, to=4-5]
	\arrow[no head, from=5-3, to=5-4]
	\arrow[no head, from=4-3, to=4-4]
	\arrow[no head, from=9-3, to=9-4]
	\arrow[no head, from=8-3, to=8-4]
	\arrow[no head, from=3-3, to=3-4]
	\arrow[no head, from=2-3, to=2-4]
	\arrow[no head, from=7-3, to=7-4]
	\arrow[no head, from=8-2, to=8-3]
	\arrow[no head, from=6-3, to=6-4]
	\arrow[no head, from=6-2, to=6-3]
	\arrow[no head, from=2-2, to=2-3]
	\arrow[no head, from=7-2, to=7-3]
	\arrow[no head, from=4-2, to=4-3]
	\arrow[no head, from=9-2, to=9-3]
	\arrow[no head, from=5-2, to=5-3]
	\arrow[no head, from=3-2, to=3-3]
\end{tikzcd}\]}
\caption{Right special tree for the recoding associated to the gapped digit substitution $\calS$}\label{fig:RecodedTree}
\end{figure}
 This right special tree happens to induce a simple $s(n)$ function:
    \begin{equation*}
    s_\calS(n) = \begin{cases}
        6, & n=1,2\\
        8, & n \ge 3
    \end{cases}
    \end{equation*}
And using the right special function, along with Proposition \ref{prop:special}, we get
\begin{theorem}
    For the sequence associated to the gapped digit tiling $\calS$, we obtain the complexity function 
    \begin{equation*}
    p_\calS(n) = \begin{cases}
        2, & n=1\\
        4, & n=2\\
        6n+14, & n = 3,4,5\\
        8n+12 & n \ge 6.
    \end{cases}
    \end{equation*}
\end{theorem}

So we see that gapped substitutions can be rewritten as constant-length substitutions, but the price we pay for that is increasing the number of symbols. This is why a substantial increase in complexity over constant-length counterparts is possible.

\section{For further reading}

If you are interested in learning the basics of the field of symbolic dynamics, we recommend both \defword{Symbolic dynamics} \cite{Kitchens} and \defword{An introduction to symbolic dynamics and coding} \cite{LM21}. These are accessible at an undergraduate
level. These books set the foundation of the field and illustrate the many connections between mathematics and theoretical computer science.

There are two books on substitution sequences that serve as the modern introduction to the field, both requiring more significant mathematical background. \defword{Substitutions in dynamics, arithmetics, and combinatorics} \cite{Fog02} is a collection of survey articles covering the wide variety of ways that substitution sequences are studied. \defword{Substitution dynamical systems -- spectral analysis} \cite{Que10} is an in-depth study rooted primarily in the spectrum of dynamical systems that come from substitutions.

Digit systems are an interesting topic of study on their own. If you're interested, see \cite{Vin00} and references therein. Gapped digit substitutions have only recently been introduced, appearing in \cite{Cab23} under the name \defword{constant-shape substitutions} and in \cite{FM22} as \defword{digit substitutions}. As such, there are no expository papers (beyond this one!) for us to recommend. Check these papers out to find out what sorts of problems are being studied today!

\end{document}